\documentclass[12pt]{article}
\usepackage{amsmath,amsthm,amsfonts,amssymb,amscd, amsxtra}
\newcommand{\banacha}{X}

\newcommand{\argmin}{argmin }
\newtheorem{theorem}{Theorem}
\newtheorem{lemma}[theorem]{Lemma}
\newtheorem{definition}{Definition}
\newtheorem{corollary}[theorem]{Corollary}
\newtheorem{proposition}[theorem]{Proposition}

\newtheorem{remark}{Remark}

\newtheorem{algorithm}{Algorithm}
\begin{document}
\title{Convergence of the Gauss-Newton method for convex composite optimization under a majorant condition}

\author{ O. P. Ferreira\thanks{IME/UFG, Campus II- Caixa
    Postal 131, CEP 74001-970 - Goi\^ania, GO, Brazil (E-mail:{\tt
      orizon@mat.ufg.br}). The author was supported in part by
    CNPq Grant 473756/2009-9, CNPq Grant 302024/2008-5, PRONEX--Optimization(FAPERJ/CNPq) and FUNAPE/UFG.}
   \and  M. L. N. Gon\c calves \thanks{IME/UFG, Campus II- Caixa
    Postal 131, CEP 74001-970 - Goi\^ania, GO, Brazil (E-mail:{\tt
      maxlng@mat.ufg.br}). The author was supported in part by CAPES and
    CNPq Grant 473756/2009-9.} \and P. R. Oliveira \thanks{COPPE-Sistemas, Universidade Federal do Rio de Janeiro,
Rio de Janeiro, RJ 21945-970, BR (Email: {\tt poliveir@cos.ufrj.br}).
This author was supported in part by CNPq.} }

\maketitle
\begin{abstract}

Under the hypothesis that an initial point is a quasi-regular point, we use a majorant condition 
to present a new semi-local convergence analysis of an extension of the Gauss-Newton method for solving  convex composite optimization problems.
In this analysis the conditions and proof of convergence are simplified by using a simple majorant condition to define  regions where a Gauss-Newton sequence is ``well behaved".

\noindent
  \textsc{AMSC: 47J15, 65H10.}

\end{abstract}
\section{Introduction}\label{sec:int}

Consider the convex composite optimation problem
\begin{equation}\label{eq:p}
\min \; h(F(x)),
\end{equation}
where  $h:\mathbb{R}^{m}\to \mathbb{R}$ is a real-valued convex and $F:\mathbb{R}^{n}\to \mathbb{R}^{m}$ is continuously differentiable.
As it is well known, see  \cite{burke,chon10,chon2002} and references therein,  a wide variety of applications with  this formulation can be found in mathematical programming literature, e.g.,  nonlinear inclusions,
penalization methods, minimax, and goal programming. Besides its practical applications, this model provides a convenient tool for the study
of first and second order  optimality conditions in constrained optimization.


The basic algorithm considered in \cite{burke,chon10,chon2002}, which is an extension of the Gauss-Newton method for solving nonlinear least square problem, will be considered in this paper.   The study of  \eqref{eq:p} is related to the convex inclusion problem
\begin{equation}\label{eq:cc}
F(x)\in C:=\{ z\in \mathbb{R}^{m} : h(z)\leq h(x), \, x\in \mathbb{R}^{m} \},
\end{equation}
because if $x_*\in \mathbb{R}^{n}$ satisfies the convex inclusion \eqref{eq:cc} then $x_*$ is a solution of \eqref{eq:p},
but if $x_*\in \mathbb{R}^{n}$ is a solution of \eqref{eq:p} it does not necessarily satisfy the inclusion  convex \eqref{eq:cc}.
Although a priori, our  goal is to  give criteria that ensure the convergence of the sequence generated by the Gauss-Newton algorithm for a solution of  \eqref{eq:p}, we  will  give a criteria that ensure the convergence of that  sequence for some $x_*\in \mathbb{R}^{n}$  satisfying  $F(x_*)\in C$ which,  in particular,  solves~\eqref{eq:p}.


In this paper, we are interested in the semi-local convergence analysis, i.e., based on the information at an initial point, criteria are given that ensure the convergence of the sequence generated by the Gauss-Newton algorithm for some $x_*\in \mathbb{R}^{n}$  with $F(x_*)\in C$. Under the  hypothesis that the initial point is a quasi-regular point of the inclusion
\eqref{eq:cc}, we use a majorant condition similar to the one used in \cite{Max2,FS02,FS01} to present  a new semi-local convergence analysis for the sequence generated by the Gauss-Newton algorithm.  The convergence analysis presented here communicates the conditions and  proof in quite a simple manner. This is possible thanks to our majorant condition and a demonstration techinique in which instead of only looking to the generated sequence, we identify regions where the Gauss-Newton sequence (for the convex composite optimation problem) is well behaved, as compared with Newton method applied to an auxiliary function associated with the majorant function. This technique was introduced in~\cite{FS02}.


The convergence of the sequence generated by the Gauss-Newton algorithm was also studied in \cite{burke,chon10,chon2002}. Among these, the criterion introduced by Li and Ng in \cite{chon10} is the best. Besides the technique used in the demonstration, the main difference from our analysis regarding \cite{chon10} is that they used
Wang's condition, introduced in  \cite{XW10}, in place of our majorant condition.
But, the formulation using the majorant condition provides a clear relationship between the majorant function and the nonlinear function $F$ under consideration. Besides this, the majorant condition simplifies the  proof of convergence.

The organization of our paper is as follows. In section \ref{aux}, we list some notation and one basic result. The Gauss-Newton algorithm is discussed  in Section \ref{sec:int.100},
   in Section \ref{sec:int.1} we present some regularity properties, and  an analysis of the
majorant and auxiliary functions is established in Section~\ref{mc}.
In Section \ref{lkant} the main result is stated and in Section \ref{sec:convxk} it is proved. Some applications of this result
are given in Section~\ref{sec:ec}.

\subsection{Notation and auxiliary result}\label{aux}

The following notation and result are used throughout our presentation.
Let $\mathbb{R}^n$ be with a norm $\| \cdot\|$.   The open and closed ball in $\mathbb{R}^n$ with center $x$ and radius $r$ are denoted, respectively by
$B(x,r) \;  \mbox{and}\;B[x,r] $.
The polar of a closed convex  $W \subset \mathbb{R}^n$ is the set
$
W^{o}:=\{z \in \mathbb{R}^n: \langle z,w \rangle\leq 0, \forall w \in W\}.
$
The distance from a point  $x$ to a set $W \subset \mathbb{R}^n$ is given by
$
d(x,W):=\inf \{\|x-w\| : w\in W \}.
$
The  set of all subsets of $\mathbb{R}^n$  is denoted by  $P(\mathbb{R}^n)$ and  $Ker(A)$ represents the kernel
 of the linear map~A.  Finally, the sum of a point  $x \in \mathbb{R}^n$ with a set $X\in P(\mathbb{R}^n) $ is the set  given by $y+\banacha=\{y+x: x\in\banacha\}$.

The following auxiliary result of elementary convex analysis will be needed:
\begin{proposition}\label{pr:conv.aux1}
Let $I\subset \mathbb{R}$ be an interval, and $\varphi:I\to \mathbb{R}$ be convex.
 If $u,v,w\in I$, $u<w$, and $u\leq v\leq w$ then
$$  \varphi(v)-\varphi(u) \leq \left[\varphi(w)-\varphi(u)\right] \frac{v-u}{w-u}.$$
\end{proposition}
\begin{proof}
See Theorem 4.1.1 on p.21 of \cite{Lem1}.
\end{proof}

\section{Preliminary } \label{sec:int.100}

In this section we present the algorithm to solve problem \eqref{eq:p}, a brief study of regularity, and
an analysis of our majorant and auxiliary functions. The results of this section are the main tools used in the proof of  convergence of the sequence generated by the
Gauss-Newton algorithm.

  In order to state the  Gauss-Newton algorithm, for solving problem \eqref{eq:p}, we need the following definition:  For  $\Delta \in (0,+\infty)$ and \(x\in \mathbb{R}^{n}\) define
 \begin{equation}\label{eq:mps}
D_{\Delta}(x):= \argmin \left\{ h(F(x)+F'(x)d): \,  d \in \mathbb{R}^n, \, \|d\| \leq \Delta \right\},
\end{equation}
 that is,   $D_{\Delta}(x)$  is   the solution set for the following  problem
\begin{equation}\label{eq:mp}
\min \left\{h(F(x)+F'(x)d): d\in \mathbb{R}^n, \, \|d\| \leq \Delta \right\}.
\end{equation}
Given that  $\Delta \in (0,+\infty]$,  $\eta \in [1,+\infty)$   and  a point    $x_0 \in \mathbb{R}^n$,  {\it the Gauss-Newton type  algorithm}  associated with \((\Delta, \eta, x_0)\)  as defined  in  \cite{burke} (see also,  \cite{chon10,chon2002})  is as follows:
\begin{algorithm} \label{algor1}
\hfill \vspace{.3cm}

\noindent
{\sc Initialization.} Take  $\Delta \in (0,+\infty]$,  $\eta \in [1,+\infty)$   and   $x_0 \in \mathbb{R}^n$. Set $k=0$.\\
{\sc Stop criterion.}  Compute $D_{\Delta}(x_k)$. If $0 \in D_{\Delta}(x_k),$  STOP. Otherwise.\\
{\sc Iterative Step.} Compute $d_k$ satisfying
$$d_k\in D_{\Delta}(x_k), \qquad \|d_k\| \leq \eta  d(0,D_{\Delta}(x_k)),$$
and set
$$x_{k+1}=x_k+d_k,$$
\(k=k+1\) and GO TO {\sc Stop criterion}.
\end{algorithm}
Note that, since  \eqref{eq:mp} is a convex optimization problem in a compact set, it follows that
 the set  $D_{\Delta}(x)$ is nonempty, for all \(x\in\mathbb{R}^n\). Therefore, the  sequence $\{x_k\}$
generated by {\bf  Algorithm~\ref{algor1}} is well defined.

\subsection{Regularity} \label{sec:int.1}
In this section  we state the hypothesis on the starting point of the sequence generated  by {\bf  Algorithm~\ref{algor1}}, which we need in our analysis, as well as some related concepts.

 Let C be as defined in \eqref{eq:cc}, that is,  C is the set of all minimum points of $h$.  For each \(x\in \mathbb{R}^n\),  we  define the set \( D_C(x)\) associated to \(C\)  as
 $$
 D_C(x):=\{d\in \mathbb{R}^n:F(x)+F'(x)d \in C\}.
 $$
In the next proposition we state a relation  between the sets   \(D_\Delta(x)\) and \(D_C(x)\).
\begin{proposition}\label{DX}
 Let  \(x\in \mathbb{R}^n\). If  $D_C(x)\neq \emptyset$ and  $ d(0,D_C(x))\leq \Delta,$  then
\[
D_\Delta(x)=\{ d\in \mathbb{R}^n:\|d\|\leq \Delta, F(x)+F'(x)d\in C\}\subset D_C(x).
\]
As a consequence,  \( d(0,D_{\Delta}(x))=d(0,D_C(x)).\)
\end{proposition}
\begin{proof}
By definition of  \(C\) in \eqref{eq:cc} and  \( D_\Delta(x)\) in \eqref{eq:mps} it can be seen that
 \[
  \{ d\in \mathbb{R}^n:\|d\|\leq \Delta, F(x)+F'(x)d\in C\}\subset D_\Delta(x).
  \]
Let \( d\in  D_\Delta(x)\).  Since $D_C(x)\neq \emptyset$ and $ d(0,D_C(x))\leq \Delta,$  there  exists  \(\bar d\in D_C(x)\) such that  \(\|\bar d\|\leq \Delta\) and \( F(x)+F'(x)\bar d\in C\).
Hence,  from the definition of  \(C\) in \eqref{eq:cc} and  \( D_\Delta(x)\) in \eqref{eq:mps}  we obtain  \(\bar d\in D_\Delta(x)\). Therefore,  as  \( \bar d,  d\in  D_\Delta(x)\), and using again the definition of \( D_\Delta(x)\) in \eqref{eq:mps},  we have
 \[
 h(F(x)+F'(x)d) = h(F(x)+F'(x)\bar d).
 \]
Now, using   \( F(x)+F'(x)\bar d\in C\), the last equality and definition of  \(C\), we obtain   \(
F(x)+F'(x)d\in C\),  which proves the first statement. The second statement, i.e., \(D_\Delta(x)\subset D_C(x)\)  can be seen by definition of $D_C(x)$.
 To conclude  the proof,  first note that the inclusion  \(D_\Delta(x)\subset D_C(x)\)  implies that
 \begin{equation} \label{eq:idx}
 d(0,D_{\Delta}(x))\geq d(0,D_C(x)).
 \end{equation}
Since $D_C(x)\neq \emptyset$ and  $ d(0,D_C(x))\leq \Delta,$ there exists \(\bar d\in D_C(x)\)  such that
 \[
 \| \bar d\|=d(0,D_C(x))\leq \Delta.
 \]
 Hence, from  definition of  \(C\) in \eqref{eq:cc} and  \( D_\Delta(x)\) in \eqref{eq:mps}  we  conclude that \(\bar d\in D_\Delta(x)\).  Therefore,
  \[
  d(0,D_{\Delta}(x))\leq \|\bar d\|=d(0,D_C(x))
  \]
 and taking into account \eqref{eq:idx}, the proof is concluded.
\end{proof}

\begin{definition}\label{qr}
Let $F:\mathbb{R}^n \to \mathbb{R}^{m}$ be a continuously differentiable
function and  let $h:\mathbb{R}^{m}\to \mathbb{R}$ be a real-valued convex function.
A point  $x_0 \in \mathbb{R}^n$  is called a quasi-regular point of the inclusion \eqref{eq:cc}, that is,  of the inclusion
\[
F(x)\in C:=\{ z\in \mathbb{R}^{m} : h(z)\leq h(x), \, x\in \mathbb{R}^{m} \},
\]
if  ${r} \in (0,+\infty)$ exists as well as an increasing positive-valued  function    $\beta : [0, {r} )\to (0, +\infty )$ such that
\begin{equation}\label{eq:dqr}
 D_C(x)\neq \emptyset,  \quad \quad d(0, D_C(x))\leq  \beta (\|x-x_0\|)d(F(x),C),  \qquad \forall x \in B(x_0,{r}).
\end{equation}
\end{definition}
Let $x_0 \in \mathbb{R}^n$  be a quasi-regular point of the inclusion \eqref{eq:cc}.  We denote $r_{x_0}$   the supremum of ${r}$ such that \eqref{eq:dqr}  holds for some increasing positive-valued function
$\beta$ on $[0,{r})$, that is,
\begin{equation} \label{eq:qrr}
r_{x_0}:=\sup \left\{ r : \exists \,  \beta : [0, r )\to (0, +\infty ) \, \mbox{ satisfying } \,  \eqref{eq:dqr} \right\}.
\end{equation}
Let $r \in [0,r_{x_0})$.  The set  ${\cal B}_{r}(x_0)$ denotes the set of all increasing positive-valued functions
$\beta$ on $[0,r)$ such that  \eqref{eq:dqr}  holds, that is,
\[
{\cal B}_{r}(x_0):=  \left\{  \beta : [0, r )\to (0, +\infty ) : \beta  \, \mbox{ satisfying } \,   \eqref{eq:dqr} \right\}.
\]
Define
\begin{equation} \label{eq:qrf}
\beta_{x_0}(t):=\inf \left\{\beta(t): \beta \in {\cal B}_{r_{x_0}}(x_0)\right\}, \quad \quad t \in [0, r_{x_0}).
\end{equation}
The number  $r_{x_0}$ and  the function $\beta_{x_0}$ are called, respectively,   the {\it quasi-regular radius} and the {\it quasi-regular bound function} of the quasi-regular point $x_0$.
\begin{remark}\label{rem:1}
Note that from the definition of  $r_{x_0}$  and  $\beta_{x_0}$ it is easy to prove that  for all \( r \leq  r_{x_0}  \)  such that  \( \lim_ { t \to r ^{-}} \beta (t)<+ \infty \) it holds that
$$
\beta_{x_0}(t)=\inf \left\{\beta(t): \beta \in {\cal B}_{r }(x_0)\right\}, \qquad t \in [0, r ).
$$
\end{remark}
\subsection{The majorant condition } \label{mc}
In this section, we define the  majorant condition for the nonlinear function $F$, which relaxes the assumption of Lipschitz
continuity to $F'$, used in our analysis.   We   present an analysis of the behavior of the majorant
 function and  of a certain associated auxiliary function -  more  details about the majorant condition can be found in  \cite{Max2,FS02,FS01}.
\begin{definition}\label{majo}
Let $R>0$, $x_0 \in \mathbb{R}^n$ and  $F:\mathbb{R}^{n}\to \mathbb{R}^{m}$ be continuously  differentiable.
 A   twice-differentiable function $f:[0,\; R)\to \mathbb{R}$  is a majorant function  for  the function  \(F\) on $B(x_0,R)$ if it satisfies
  \begin{equation}\label{KH}
  \|F'(y)-F'(x)\| \leq f'(\|y-x\|+\|x-x_0\|)-f'(\|x-x_0\|),
  \end{equation}
   for any $x,y\in B(x_0, R)$,  $\|x-x_0\|+\|y-x\|< R$,  and moreover,
  \begin{itemize}
  \item[{\bf h1)}] $f(0)=0$, $f'(0)=-1$;
  \item[{\bf h2)}] $f'$ is convex and strictly increasing.
  \end{itemize}
\end{definition}
In the next result we  bound the linearization error of  the function $F$  by the error in the linearization on the  majorant function.
\begin{lemma}  \label{pr:taylor}
Take
$$
   x,y\in B(x_0, R) \quad\mbox{and}\quad 0\leq t<v<R.
$$
If $\|x-x_0\|\leqslant t$ and $\|y-x\|\leqslant v-t$, then

\[
\|F(y)-[ F(x)+F'(x)(y-x)\|\leqslant f(v)-[f(t) +f'(t)(v-t)]\left(\frac{\|y-x\|}{v-t}\right)^2.
\]
\end{lemma}
\begin{proof}
The proof follows the same pattern as Lemma 7 from \cite{FS01}.
 \end{proof}
To state our main theorem we need  a certain auxiliary function associated with the majorant function.
We shall see later that the  sequence generated by {\bf Algorithm \ref{algor1}} will be ``majorized " by the Newton sequence associated with this auxiliary function.

Let   $f:[0,\; R)\to \mathbb{R}$  be a  majorant
function for the function $F$ on  $B(x_0, R)$. Take $ \xi>0$,  $\alpha>0$  and define  the { \it auxiliary  function}
  \begin{equation} \label{defh}
  \begin{array}{rcl}
f_{ \xi, \alpha} :[0, R)&\to& \mathbb{R}\\
    t&\mapsto& \xi+(\alpha-1)t+\alpha f(t).
  \end{array}
\end{equation}
Now,  consider the following conditions on the auxiliary  function   $ f_{ \xi,\alpha}$:
    \begin{itemize}
  \item[{\bf h3)}]  there exists  $t_* \in (0,  R)$ such that   $f_{ \xi, \alpha}(t)>0$   for all  $t \in (0,  t_*)$ and $f_{ \xi, \alpha}(t_*)=0$;
    \end{itemize}
  \begin{itemize}
  \item[{\bf h4)}] $f_{\xi, \alpha}'(t_*)<0$.
  \end{itemize}
From now on, we assume that  $f:[0,\; R)\to \mathbb{R}$ is a  majorant
function for the function $F$ on  $B(x_0, R)$ and that {\bf h3} holds.  The assumption {\bf h4} will be considered to hold only when explicitly stated.
\begin{proposition} \label{pr:10} The following statements  hold:
\begin{itemize}
 \item[{\bf i)}]  $f_{\xi,\alpha}(0)=\xi >0$,   $f'_{\xi, \alpha}(0)=-1$;
  \item[{\bf  ii)}]  $f'_{\xi,\alpha}$ is convex and strictly increasing.
  \end{itemize}
\end{proposition}
\begin{proof}
Seen from the definition in \eqref{defh} and assumptions  {\bf h1} and {\bf h2}.
\end{proof}
\begin{proposition} \label{pr:1}
 The function $f_{ \xi, \alpha}$  is strictly convex, and
  \begin{equation} \label{eq:n.f}
    f_{ \xi, \alpha}(t)>0, \quad f_{ \xi, \alpha}'(t)<0, \qquad t<t-f_{ \xi, \alpha}(t)/f_{ \xi, \alpha}'(t)< t_*,
    \qquad\qquad  \forall ~t\in [0,t_*) .
  \end{equation}
  Moreover,  $f_{ \xi, \alpha}'(t_*)\leq 0$.
\end{proposition}
\begin{proof}
Using Proposition \ref{pr:10}, the proof follows the same pattern as Proposition 3 from \cite{FS01}.
\end{proof}
In view of the second inequality in (\ref{eq:n.f}), the Newton iteration map is
well defined in $[0,t_*)$.
 Let us call it
\begin{equation} \label{eq:n.f.2}
  \begin{array}{rcl}
  n_{f_{\xi, \alpha}}:[0,t_*)&\to& \mathbb{R}\\
    t&\mapsto& t-f_{ \xi, \alpha}(t)/f_{\xi,\alpha}'(t).
  \end{array}
\end{equation}
\begin{proposition} \label{pr:1.5}
  For  each $ t\in [0,t^*)$ it holds that $\xi \leq n_{f_{ \xi, \alpha}}(t)$.
  \end{proposition}
\begin{proof}
Proposition~\ref{pr:1} implies that $f_{\xi, \alpha}$ is convex. Hence, using the first item of Proposition~\ref{pr:10}
 it is easy to see, by using convexity properties,  that $t-\xi \geq - f_{ \xi, \alpha}(t)$. So,  the above definition  implies that
$$
 n_{f_{\xi, \alpha}}(t)-\xi=t- \frac{f_{\xi, \alpha}(t)}{f'_{\xi, \alpha}(t)}-\xi\geq -f_{\xi, \alpha}(t)- \frac{f_{\xi, \alpha}(t)}{f'_{\xi, \alpha}(t)}=\frac{f_{\xi, \alpha}(t)}{-f'_{\xi, \alpha}(t)}[f'_{\xi, \alpha}(t)+1], \qquad  \forall ~t\in [0,t_*) .
$$
Proposition~\ref{pr:10} implies that $f'_{\xi, \alpha}(0)=-1$ and   $f'_{\xi,\alpha}$ is strictly increasing. Thus, we obtain  $f'_{\xi, \alpha}(t)+1\geq 0$, for all $t\in [0,t_*)$.
Therefore,  combining the above inequality with the first two inequalities in Proposition~\ref{pr:1} the desired result follows.
\end{proof}
\begin{proposition} \label{pr:2}
  Newton iteration map  $n_{f_{ \xi, \alpha}}$  maps $[0,t^*)$ in
  $[0,t^*)$, and it holds that
  $$  t<n_{f_{ \xi, \alpha}}(t), \qquad    t_*-n_{f_{ \xi, \alpha}}(t)\leqslant  \frac{1}{2}(t_*-t) ,\qquad
    \forall\,t \in [0,t_*).$$
If ${f_{ \xi, \alpha}}$ also satisfies {\bf h4}, i.e., ${f_{ \xi, \alpha}'}(t_*)<0$, then
$$
t_*-n_{f_{ \xi, \alpha}}(t) \leq \frac{f_{ \xi, \alpha}''(t_*)}{-2f_{\xi,\alpha}'(t_*)}  (t_*-t)^2, \qquad \forall\,t\in [0,t_*).
$$
\end{proposition}
\begin{proof}
  The proof follows the same pattern as Proposition 4 of \cite{FS01}.
\end{proof}
The Newton sequence  $\{t_{k}\}$  for solving the equation  $f_{\xi, \alpha}(t)=0$  with starting point $t_0=0$   is defined as
\begin{equation}
  \label{eq:tknk}
  t_{0}=0,\quad t_{k+1}=n_{f_{\xi, \alpha}}(t_{k}), \qquad k=0,1,\ldots\, .
\end{equation}
Therefore, by also using Proposition~\ref{pr:2} it is easy to prove  that
\begin{corollary} \label{cr:kanttk}
  The sequence $\{t_{k}\}$ is well defined, is strictly increasing,
  is contained in $[0,t_*)$, and converges $Q$-linearly to $t_*$  as follows
$$
 t_*-t_{k+1}\leq \frac{1}{2}( t_*-t_{k}), \qquad   k=0,1,\,\ldots\, .
$$
If ${f_{ \xi, \alpha}}$ also satisfies assumption {\bf h4}, then $\{t_{k}\}$ converges $Q$-quadratically to $t_*$  as follows
$$
       t_*-t_{k+1} \leq \frac{f_{\xi,\alpha}''(t_*)}{-2 f'_{\xi,\alpha}(t_*)}(t_*-t_{k})^2,
    \quad  k=0,1,\dots\,.
  $$
\end{corollary}

\begin{proposition}\label{fcdecrescente}
The map $[0,t_*) \ni t \to -f_{ \xi, \alpha}(t)/f_{ \xi, \alpha}'(t)$ is  decreasing.
\end{proposition}
\begin{proof}
Proposition~\ref{pr:1} implies that   $f_{ \xi, \alpha}'(t)\neq 0$ for all $t\in [0,t_*)$. So, the function in the proposition is well defined.
As $f_{ \xi, \alpha}$ is twice-differentiable we have
$$\left(\frac{-f_{ \xi, \alpha}(t)}{f_{ \xi, \alpha}'(t)}\right)'=\frac{f_{ \xi, \alpha}(t)f_{ \xi, \alpha}''(t)-(f_{ \xi, \alpha}'(t))^2}{(f_{ \xi, \alpha}'(t))^2}, \qquad \forall \; t \in [0,t_*).$$
Hence, it suffices to show that
  \begin{equation}\label{eq:aux.pr.10}
  f_{ \xi, \alpha}(t)f_{ \xi, \alpha}''(t)-(f_{ \xi, \alpha}'(t))^2 \leq 0,\qquad \forall \; t \in [0,t_*).
    \end{equation}
Since $f_{ \xi, \alpha}$ is strictly convex (Proposition \ref{pr:1}) and $f_{ \xi, \alpha}'$ is  convex (Proposition \ref{pr:10}), we have
 $$
        0> f_{ \xi, \alpha}(t)+f_{ \xi, \alpha}'(t)(t_*-t),\quad f_{ \xi, \alpha}''(t)\geq 0, \quad f_{ \xi, \alpha}'(t_*)\geq f_{ \xi, \alpha}'(t)+f_{ \xi, \alpha}''(t)(t_*-t), \quad \forall \; t\in [0,t_*).
$$
Using these inequalities and  the second inequality in \eqref{eq:n.f}, we obtain
$$
f_{ \xi, \alpha}(t)f_{ \xi, \alpha}''(t)-(f_{ \xi, \alpha}'(t))^2 \leq f_{ \xi, \alpha}'(t)(t-t_*)f_{ \xi, \alpha}''(t)-(f_{ \xi, \alpha}'(t))^2\\
\leq-f_{ \xi, \alpha}'(t)f_{ \xi, \alpha}'(t_*),
$$
which combined with Proposition \ref{pr:1} yields the inequality in \eqref{eq:aux.pr.10}. Therefore, the proposition is fulfilled.
\end{proof}
\begin{proposition} \label{nnn}
It holds that  $\xi <t_*$.  Moreover, if
\begin{equation}\label{alfa}
\alpha\geq \frac{\eta\beta_{x_0}(t)}{\eta\beta_{x_0}(t)[f'(t)+1]+1}, \qquad \forall\; \xi\leq t< t_*,
\end{equation}
then
$$
{\eta\beta_{x_0}{(t)}}/{\alpha}\leq -{1}/{f'_{ \xi, \alpha}(t)}, \qquad \forall\; \xi\leq t< t_*.
$$
\end{proposition}
\begin{proof}
Proposition \ref{pr:1} implies  that $f_{ \xi, \alpha}$ is strictly convex, which combined with the definition of $t_*$ in {\bf h3} and item {\bf i}  of  Proposition~\ref{pr:10}   gives
$$
0=f_{ \xi, \alpha}(t_*)>f_{ \xi, \alpha}(0)+f_{ \xi, \alpha}'(0)(t_*-0)=  \xi-t_*,
$$
which proves the first statement.

Combining the assumption in \eqref{alfa}, as well as {\bf  h1}  and{ \bf h2}, we obtain after simple calculus that
$$
\alpha {\eta\beta_{x_0}(t)(f'(t)+1)+\alpha}\geq \eta\beta_{x_0}(t), \qquad \forall\; \xi \leq t< t_*.
$$
Hence, using $f'_{\xi,\alpha }(t)=(\alpha-1)+\alpha f'(t)$ and some algebraic manipulations, the last inequality becomes
$$
 \eta\beta_{x_0}(t)f'_{\xi,\alpha }(t) \geq -\alpha, \qquad \forall\; \xi\leq t< t_*,
$$
which combined with the second inequality in \eqref{eq:n.f} yields the desired inequality.
\end{proof}
\begin{proposition}\label{pr:111}
Let $0<\bar{\alpha}<\alpha$ for the corresponding auxiliary functions $f_{\xi,\bar{\alpha}}$ and $f_{\xi,\alpha}$, as well as  ${\bar t}_{*}$ and $ t_{*}$,   its smallest  zeros, respectively. Then the following assertions hold:
\begin{itemize}
 \item[{\bf i)}] $f_{ \xi, \bar{\alpha}}< f_{\xi,{\alpha}}$ on $(0, R)$;
 \item[{\bf ii)}]  $f_{ \xi, \bar{\alpha}}'<f_{\xi,{\alpha}}'$ on $(0, R)$;
\item[{\bf iii)}]  $ {\bar t}_{*}<t_*$.
\end{itemize}
\end{proposition}
\begin{proof}
From   {\bf h2} it follows that   $f'$  is strictly increasing which implies that $f$ is strictly convex. Thus, using  $ {\bf h1}$ we conclude  that
  $  f(t)+t> 0,$ for all $ t\in (0, R)$  and hence  the assumption  $\alpha>\bar{\alpha}$ implies
 $$
   \bar{\alpha} (t+f(t))< \alpha(t+f(t)),  \qquad \forall \; t \in [0, R).
  $$
To conclude the proof of item {\bf i}, add  $\xi -t$  on both sides of the last inequality and use the definition in \eqref{defh}.

To prove item {\bf ii},   we first use that $f'$ is strictly increasing ($ {\bf h2}$), as well as the assumption  $\alpha>\bar{\alpha}$ to obtain that  $ (\alpha-\bar{\alpha})(f'(t)-f'(0))> 0,$
for all $ t\in (0, R).$ Hence, from $ {\bf h1}$ and some algebraic manipulation, we obtain
 \[
    (\bar{\alpha}-1)+\bar{\alpha}f'(t)< (\alpha-1)+\alpha  f'(t),  \qquad \forall \; t \in [0, R).
 \]
 So, by using the  definition in \eqref{defh}, the statement holds true.

To establish item {\bf iii}, use  item {\bf i}  and the definition of ${\bar t}_{*}$ and $t_*$ in~{\bf h3}.
\end{proof}
\section{Semi-local analysis for the Gauss-Newton method } \label{lkant}
In this section our goal is to state and prove a semi-local theorem for  the  sequence generated by   {\bf Algorithm \ref{algor1}}
in order to solve problem \eqref{eq:p}. Under the  hypothesis that the initial point is a quasi-regular point of the inclusion
\eqref{eq:cc} and the nonlinear function \(F\) satisfies  the  majorant condition in Definition~\ref{majo}, we will prove
convergence  of the sequence to a point $x_*\in B[x_0, t_*]$ such that  $F(x_*)\in C,$ and in particular that $x_*$ solves \eqref{eq:p}.  The statement of the theorem~is:
\begin{theorem}\label{th:knt}
Let $F:\mathbb{R}^n \to \mathbb{R}^{m}$ be a continuously differentiable function.  Assume  that  $R>0$,  $x_0\in\mathbb{R}^n$ and $f:[0,\;  R)\to \mathbb{R}$   is a  majorant
function for $F$ on  $B(x_0,  R).$   Take  the constants
\(
 \alpha >0
\)
and
\(
 \xi >0
\)
and consider the auxiliary
function \mbox{$f_{ \xi, \alpha}:[0, R)\to \mathbb{R}$},
$$
  f_{ \xi, \alpha}(t):=\xi+(\alpha-1)t+\alpha f(t).
$$
If    $f_{ \xi, \alpha}$   satisfies {\bf h3}, i.e.,  $t_*$  is the smallest zero of $f_{ \xi, \alpha}$,  then the sequence
  generated by Newton's method  for solving $ f_{ \xi, \alpha}(t)=0,$  with starting
  point $t_{0}=0$,
\begin{equation}\label{ns.KT}
    t_{k+1} ={t_{k}}-f_{ \xi, \alpha}'(t_{k}) ^{-1}f_{ \xi, \alpha}(t_{k}),\quad k=0,1,\ldots\,,
\end{equation}
  is well defined, $\{t_{k}\}$ is strictly increasing, is contained in
  $[0,t_*)$, and converges $Q$-linearly to $t_*$.
  Let   \(\eta \in [1,\infty)\),  \( \Delta \in (0,\infty]\) and   $h:\mathbb{R}^{m}\to \mathbb{R}$  a real-valued convex function with minimizer set \(C\)  nonempty.
Suppose that  $x_0 \in \mathbb{R}^n$ is a
quasi-regular point of the inclusion
\[
F(x)\in C,
\]
with the quasi-regular radius $r_{x_0}$ and
the quasi-regular bound function~$\beta_{x_0}$  as defined in \eqref{eq:qrr} and  \eqref{eq:qrf}, respectively.
 If    \(d(F(x_0),C) >0\),   $t_* \leq r_{x_0}$,
 \begin{equation}\label{a11}
  {\Delta}\geq \xi \geq \eta\beta_{x_0}(0)d(F(x_0),C), \qquad \alpha \geq  \sup\left\{\frac{\eta\beta_{x_0}(t)}{\eta\beta_{x_0}(t)[f'(t)+1]+1}:  \xi \leq t < t_*\right\},
  \end{equation}
  then the sequence generated by {\bf Algorithm \ref{algor1}}, denoted by $\{x_k\},$ is contained in $B(x_0,
  t_*)$,
\begin{equation}\label{eq:001}
F(x_k)+F'(x_k)(x_{k+1}-x_{k}) \in C,  \quad k=0,1, \ldots\, ,
\end{equation}
satisfies the inequalities
\begin{equation}\label{eq:bd}
  \|x_{k+1}-x_{k}\|   \leq  t_{k+1}-t_{k} , \qquad \|x_{k+1}-x_{k}\|\leq   \frac{t_{k+1}-t_{k}}{(t_{k}-t_{k-1})^2} \|x_k-x_{k-1}\|^2,
  \end{equation}
 \( k=0, 1, \ldots\, , \) and \( k=1,2, \ldots\, \), respectively,  converge  to a point $x_*\in B[x_0, t_*]$ such that  $F(x_*)\in C,$
\begin{equation}\label{eq:002}
 \|x_*-x_{k}\|   \leq  t_*-t_{k}, \qquad k=0,1, \ldots\,
\end{equation}
and  the convergence is  $R$-linear. If, additionally,   $f_{\xi, \alpha}$ satisfies  {\bf h4}    then the sequences $\{t_k\}$ and  $\{x_k\}$ converge $Q$-quadratically  and $R$-quadratically
 to $t_*$ and $x_*$, respectively.
\end{theorem}
\begin{remark}
If,
$$\alpha>\bar{\alpha}:=\sup\left\{\frac{\eta\beta_{x_0}(t)}{\eta\beta_{x_0}(t)[f'(t)+1]+1}:  \xi \leq t < t_*\right\},  $$
 then the sequence $\{x_k\}$ converges R-quadratically  to $x_*$.
To prove this assertion,  note that through item {\bf iii} of Proposition~\ref{pr:111}, we have ${\bar t}_{*}<t_{*}$. Hence, using that
 $f_{\xi,\bar{\alpha}}'$ strictly increasing  and item~{\bf ii} of Proposition~\ref{pr:111}, we obtain
$$f_{\xi,\bar{\alpha}}'( {\bar t}_{*})<f_{\xi,\bar{\alpha}}'( t_{*})<f_{ \xi,\alpha}'( t_{*}),$$
which, combined with  Proposition \ref{pr:1} implies  that  $f_{\xi,\bar{\alpha}}'( t_{\bar*})<0$. So, the statement
is correct if $f_{\xi,\alpha}$ is replaced by $f_{\xi,\bar{\alpha}}$ in Theorem \ref{th:knt}.
\end{remark}

Remember that all statements  made in Theorem~\ref{th:knt} for the sequence $t_k$  were proven in Corollary \ref{cr:kanttk}.

From now on, we assume that the hypotheses of Theorem \ref{th:knt}
hold, with the exception of {\bf h4}, which will be considered to hold
only when explicitly stated.

\subsection{Proof of convergence}\label{sec:convxk}
In this section we prove  convergence of
the sequence $\{x_k\}$ generated  by {\bf Algorithm \ref{algor1}} for solving  \eqref{eq:p}, based on the assumptions stated in
Theorem~\ref{th:knt}.

As we saw in Section \ref{sec:int.100},  $D_\Delta (x) \neq \emptyset$ for all $x\in \mathbb{R}^n$, therefore  the sequence $\{x_k\}$ is well defined.
But this is not enough  to prove the convergence of sequence $\{x_k\}$ to some point $x_*\in \mathbb{R}^n$  such that $F(x_*)\in C$, because we have no relationship between
the set of search directions  $D_\Delta (x)$ to the set of solutions of the linearized inclusion
$$F(x)+F'(x)d \in C, \qquad \|d\|\leq \Delta.$$
Now, if we prove that  $D_\Delta (x) \subset D_C(x)$ for suitable points, then  we can use  the results of regularity to relate the sets mentioned above.
First, we define some subsets of $B(x_0, t_*)$ in which, as we shall
prove, the desired inclusion holds   for all points in these subsets.  
\begin{align}\label{E:K}
K(t)&:=\left\{ x\in\mathbb{R}^n \, : \;  \|x-x_0\|\leq t, ~ \eta d(0,D_C(x)) \leqslant -\frac{f_{ \xi, \alpha}(t)}{f_{ \xi, \alpha}'(t)}\right\},\qquad
   t\in [0,t_*)\,,\\
  \label{eq:def.K}
 K&:=\bigcup_{t\in[0,t_*)} K(t).
\end{align}
In \eqref{E:K}  we assume that $0\leqslant t<t_*$, therefore it follows from  Proposition \ref{pr:1}  that  $f'_{\xi,\alpha}(t)\neq 0$.
So, the above definitions are consistent.


\begin{proposition} \label{sss}
If $x\in K$,  then
 $$
 D_\Delta(x)=\{ d\in \mathbb{R}^n~:~ F(x)+F'(x)d\in C,~\|d\|\leq \Delta \} \subset D_C(x),
 $$
 and
 $$
  d(0,D_{\Delta}(x))=d(0,D_C(x)).
 $$
\end{proposition}
\begin{proof}
From Proposition~\ref{DX} it is sufficient to prove that \(D_C(x)\neq \emptyset \) and \(d(0, D_C(x))\leq \Delta\) for all \(x\in K\).
Let  $x\in K$,  then $x\in K(t)$ for some $t\in [0,t_*)$ which implies that \(x\in B(x_0,t_*)\).  Since \(t_*\leq r_{x_0}\) and \(x_0\) is a quasi-regular point, it follows from Definition~\ref{qr} and  the definition of the quasi-regular radius in \eqref{eq:qrr} that $D_C(x)  \neq \emptyset$.

By hypothesis $\eta \geq 1$ and \(\xi\leqslant \Delta\).  Thus,  as $x\in K(t)$, by using the  definition  in   \eqref{E:K},
Proposition~\ref{fcdecrescente}, and  Proposition \ref{pr:10}   we obtain
$$d(0,D_C(x))\leqslant \eta d(0,D_C(x))
\leqslant -\frac{f_{ \xi, \alpha}(t)}{f_{ \xi, \alpha}'(t)}\leqslant -\frac{f_{ \xi, \alpha}(0)}{f_{ \xi, \alpha}'(0)}=\xi\leqslant \Delta,$$
which proves the desired result.
\end{proof}
For each  $x\in \mathbb{R}^n$, we define the set $\bar{D}_{\Delta}(x)$  as
\begin{equation}\label{def11}
\bar{D}_{\Delta}(x):= \left\{  d \in D_{\Delta}(x)~:~ \|d\| \leq \eta d(0,D_{\Delta}(x)) \right\}.
\end{equation}
As $D_\Delta (x) \neq \emptyset$ for all $x\in \mathbb{R}^n$, we have  $\bar{D}_\Delta (x) \neq \emptyset$ for all $x\in \mathbb{R}^n$
and consequently, the Gauss-Newton iteration multifunction is well defined.
Let us call  $G_{F}$   the Gauss-Newton
iteration multifunction for $F$ in $B(x_0, t_*)$:
\begin{equation} \label{NF}
  \begin{array}{rcl}
  G_{F}:B(x_0, t_*)&\to& P(\mathbb{R}^n)\\
    x&\mapsto& x+\bar{D}_\Delta(x).
  \end{array}
\end{equation}
We shall prove that the Gauss-Newton iteration multifunction  is ``well behaved'' on the subsets defined in \eqref{E:K}, but first we need the following  technical  result:
\begin{lemma} \label{First}
For each $t\in [0, t_*)$, $x\in K(t)$ and $y \in G_F(x)$ it holds that:
\begin{itemize}
 \item[{\bf i)}]  $ \|y-x\| \leq n_{f_{ \xi, \alpha}}(t)-t$;
 \item[{\bf ii)}]  $\|y-x_0\| \leq n_{f_{ \xi, \alpha}}(t)<t_*$;
 \item[{\bf iii)}]  $\displaystyle \eta d(0,D_{C}(y))\leq -\frac{f_{\xi,\alpha}(n_{f_{\xi,\alpha}}(t))}{f'_{\xi,\alpha}(n_{f_{\xi,\alpha}}(t))}\left( \frac{\|y-x\|}{n_{f_{\xi,\alpha}}(t)-t}\right)^2.$
\end{itemize}
\end{lemma}
\begin{proof}
Since  $t\in[0,t_*)$ and $x\in K(t)$,  by using the definition in \eqref{E:K}, Proposition \ref{sss}  and the
first two statements in Proposition~\ref{pr:2}, we obtain
\begin{equation} \label{eq:eq.aux.k}
  \| x-x_0\|\leq t, \qquad    \eta d(0,D_{\Delta}(x))=\eta d(0,D_C(x)) \leqslant -\frac{f_{ \xi, \alpha}(t)}{f_{ \xi, \alpha}'(t)}, \qquad t<n_{f_{\xi,\alpha}}(t)<t_*.
\end{equation}
Now, as  $y \in G_F(x)$  there exists  $d \in \bar{D}_\Delta(x)$  such that  $y=x+d$. Using the definition of the set $ \bar{D}_\Delta(x)$ in \eqref{def11}
and  the second  inequality in \eqref{eq:eq.aux.k} it  follows that
$$
\|d\|\leq \eta d(0,D_{\Delta}(x))= \eta d(0,D_{C}(x)) \leq -f_{ \xi, \alpha}(t)/f_{ \xi, \alpha}'(t).
$$
Since $d=y-x$,  the  last  inequality  together with the  definition in \eqref{eq:n.f.2} implies item {\bf i}.

Triangular inequality combined with the first  inequality in \eqref{eq:eq.aux.k},  item {\bf i}, and the last inequality in \eqref{eq:eq.aux.k} yields
\begin{equation}\label{eq:110}
\|y-x_0\| \leq \| y-x\|+ \|x-x_0\| \leq n_{f_{ \xi, \alpha}}(t)<t_*,
\end{equation}
which proves item {\bf ii}.

Since  \(\|y-x_0\|  <t_*\) and     $ t_*\leq r_{x_0}$ we obtain by the quasi regularity assumption
$$
D_C(y)\neq \emptyset,  \qquad d(0,D_C(y))\leq \beta_{x_0}(\|y-x_0\|) d(F(y),C).
$$
As $x\in  K(t)\subset K $ and $y-x=d \in D_\Delta(x)$,  it follows from Proposition \ref{sss} that
\[
F(x)+F'(x)(y-x)\in C.
\]
Therefore,  taking into account   that $\eta \geq 1$, by using the above inequality and the last  inclusion it is easy to conclude that
$$
\eta d(0,D_C(y))\leq  \eta \beta_{x_0}(\|y-x_0\|)\|F(y)-F(x)-F'(x)(y-x)\|.
$$
On the other hand, from item {\bf i} we have  $\|y-x\|\leq n_{f_{\xi,\alpha}}(t)-t$
and,  as  $\| x-x_0\|\leq t$, by using  Lemma \ref{pr:taylor}  we have
$$
\|F(y)-F(x)-F'(x)(y-x)\|\leq [f(n_{f_{\xi,\alpha}}(t))-f(t)-f'(t)(n_{f_{\xi,\alpha}}(t)-t)]\left( \frac{\|y-x\|}{n_{f_{\xi,\alpha}}(t)-t}\right)^2.
$$
Hence, combining  the two above inequalities  we conclude that
\begin{equation} \label{eq:mtae}
    \eta d(0,D_C(y))\leq  \eta \beta_{x_0}(\|y-x_0\|)[f(n_{f_{\xi,\alpha}}(t))-f(t)-f'(t)(n_{f_{\xi,\alpha}}(t)-t)]\left( \frac{\|y-x\|}{n_{f_{\xi,\alpha}}(t)-t}\right)^2.
\end{equation}
Now, the definition in \eqref{eq:n.f.2} implies  that $f_{\xi,\alpha}(t)+f'_{\xi,\alpha}(t)(n_{f_{\xi,\alpha}}(t)-t)=0$. So, we have
$$
f_{\xi,\alpha}(n_{f_{\xi,\alpha}}(t))=f_{\xi,\alpha}(n_{f_{\xi,\alpha}}(t))-f_{\xi,\alpha}(t)-f'_{\xi,\alpha}(t)(n_{f_{\xi,\alpha}}(t)-t)
$$
By using the definition in  \eqref{defh} and after simple algebraic manipulation, the last equality becomes
$$
f_{\xi,\alpha}(n_{f_{\xi,\alpha}}(t))=\alpha\left(f(n_{f_{\xi,\alpha}}(t))-f(t)-f'(t)(n_{f_{\xi,\alpha}}(t)-t)\right).\\
$$
So,  as $\beta_{x_0}$ is an increasing function,  by a simple combination of \eqref{eq:110},  \eqref{eq:mtae} and the last equality, we obtain
$$
   \eta d(0,D_{C}(y))\leq   \frac{\eta\beta_{x_0}(n_{f_{\xi,\alpha}}(t))}{\alpha}f_{\xi,\alpha}(n_{f_{\xi,\alpha}}(t))\left( \frac{\|y-x\|}{n_{f_{\xi,\alpha}}(t)-t}\right)^2.$$
From Proposition~\ref{pr:1.5} and the first statement in Proposition~\ref{pr:2} we have  $\xi \leq n_{f_{\xi,\alpha}}(t)<t_*$. Thus, by using the last inequality  and  Proposition~\ref{nnn},  the last  inequality of  the lemma follows.
\end{proof}
In the next result we prove  the desired result, namely,  that the Gauss-Newton iteration multifunction  is ``well behaved'' on the subsets defined in \eqref{E:K}.
\begin{lemma} \label{NfNF}
For each $t\in [0, t_*)$,  the following inclusions  hold: $ K(t)\subset B(x_0,t_*)$ and
$$
G_F\left( K(t) \right)\subset K\left( n_{f_{\xi,\alpha}}(t) \right).
$$
As a consequence,
$K\subset B(x_0,t_*)$ and $G_F(K)\subset K.$
\end{lemma}
\begin{proof}
The first inclusion follows trivially from the definition of $K(t)$.  Take $x\in K(t)$ and  $y \in G_F(x)$.
Combining items   {\bf i} and {\bf iii} of Lemma~\ref{First} we have
$$
     \eta d(0,D_{C}(y))\leq -\frac{f_{\xi,\alpha}(n_{f_{\xi,\alpha}}(t))}{f'_{\xi,\alpha}(n_{f_{\xi,\alpha}}(t))}.
$$
The last inequality together with item {\bf ii}  of Lemma~\ref{First}  and the definition in \eqref{E:K} show us  that $y\in
K(n_{f_{\xi,\alpha}}(t))$, which proves the second inclusion.

The next inclusion, first on the second sentence, follows trivially
from definitions \eqref{E:K} and \eqref{eq:def.K}.  To verify the last
inclusion, take $x\in K$.  Therefore, $x\in K(t)$ for some $t\in [0,t_*)$.
Using the  first part of the lemma, we conclude that $G_F(x)\subset
K(n_{f_{\xi,\alpha}}(t))$. To end the proof, note that $n_{f_{\xi,\alpha}}(t)\in [0,t_*)$ and use
the definition of $K$.
\end{proof}

Finally, we are ready to prove the main result of this section which
is an immediate consequence of the latter results. First, note that definitions \eqref{def11} and \eqref{NF} imply that the
sequence $\{x_k\}$   satisfies
\begin{equation} \label{NFS}
x_{k+1}\in G_F(x_k),\qquad k=0,1,\ldots \,,
\end{equation}
which is indeed an equivalent definition of this sequence.
\begin{corollary}\label{auxcor}
The sequence $\{x_k\}$ which is contained in $B(x_0,t_*)$,
converges to a point $x_*\in B[x_0,t_*]$ such that  $F(x_*)\in C$.  Moreover,
  $\{x_k\}$ and $\{t_{k}\}$
 satisfy \eqref{eq:001}, \eqref{eq:bd} and \eqref{eq:002}. Furthermore,  if ${f_{ \xi, \alpha}}$
also satisfies assumption {\bf h4}  then $\{x_{k}\}$  converges $R$-quadratically to $x_*$.
\end{corollary}
\begin{proof}
Since $x_0\in B(x_0,t_*)\subseteq B(x_0,r_{x_0})$; by using the quasi regularity assumption,
 $\eta \geq 1$, the first inequality in \eqref{a11}, and Proposition \ref{pr:10}; we obtain
$$
 D_{C}(x_0)\neq \emptyset, \qquad  \eta d(0,D_{C}(x_0))\leq  \eta \beta_{x_0}(0)d(F(x_0),C)\leq\xi=-\frac{f_{ \xi, \alpha}(0)}{f_{ \xi, \alpha}'(0)}.$$
Therefore,
\[  x_0\in K(0)\subset K,\]
where the second inclusion follows trivially from \eqref{eq:def.K}.
Using the above inclusion, the inclusions $G_F(K)\subset K$ (Lemma \ref{NfNF}) and   \eqref{NFS}, we conclude that  the sequence
$\{x_k\}$   rests in $K$ and, in particular, we have $\{x_k\}$ contained in $B(x_0, t_*)$. Since $\{x_k\}\subset K$,  by combining   Proposition~\ref{sss} and  {\bf Algorithm \ref{algor1}}, the inclusion in \eqref{eq:001} follows.
Now, we  prove  by induction that
\begin{equation}
        \label{eq:xktk}
        x_k\in K(t_{k}), \qquad k=0,1,\ldots \,.
\end{equation}
The above inclusion, for  $k=0$, is the first result in this proof.
Assume that $x_k\in K(t_{k})$.   From  \eqref{eq:tknk} we have $t_{k+1}=\eta_{f_{\xi, \alpha}}(t_k)$ and, as $x_k\in K(t_{k})$,  Lemma \ref{NfNF} implies that $G_F(x_k)\subset K(t_{k+1})$, which taking into account  \eqref{NFS} completes the induction  proof.

Simple combination  of {\bf Algorithm \ref{algor1}} with
\eqref{eq:xktk}, Proposition \ref{sss},    and  \eqref{E:K}  yields
\begin{equation} \label{eq:itf}
\|x_{k+1}-x_k\|\leq \eta d(0,D_\Delta(x_k))= \eta d(0,D_{C}(x_k))
\leqslant -\frac{f_{ \xi, \alpha}(t_{k})}{f_{ \xi, \alpha}'(t_{k})}, \qquad k=0,1,\ldots\,,
\end{equation}
which,  using  \eqref{ns.KT}  becomes
$$
         \|x_{k+1}-x_k\|\leq t_{k+1}-t_{k}, \qquad k=0,1,\ldots\,.
$$
So, the first inequality in \eqref{eq:bd} holds. On the other hand, as $\{t_{k}\}$ converges to $t_*$, the above inequalities imply
that
\[
   \sum_{k=k_0}^\infty \|x_{k+1}-x_k\|\leq
   \sum_{k=k_0}^\infty t_{k+1}-t_{k} =t_*-t_{k_0}<+\infty,
\]
for any $k_0\in\mathbb{N}$. Hence, $\{x_k\}$ is a Cauchy sequence in
$B(x_0, t_*)$ and so converges to some $x_*\in B[x_0,t_*]$. Moreover, the above inequality also implies \eqref{eq:002}, i.e., $\|x_*-x_k\|\leq
t_*-t_{k}$, for any $k$.
As C is closed,  $\{x_k\}$ converges to $x_*$,
$$
F(x_k)+F'(x_k)(x_{k+1}-x_k) \in C,
$$
 and $F$ is a continuously differentiable function; therefore, we have $F(x_*) \in C$.

In order to prove  the second  inequality in \eqref{eq:bd}, first note that $  x_k\in K(t_{k})$ and $t_{k+1}=n_{f_{\xi,\alpha}}(t_k)$, for all $k=0,1,\ldots .$
Thus,  take an arbitrary $k$ and apply item~{\bf iii} of  Lemma~\ref{First}  with $y=x_k$, $x=x_{k-1}$  and $t=t_{k-1}$
to obtain
$$
\eta d(0,D_{C}(x_k))\leq -\frac{f_{\xi,\alpha}(t_k)}{f'_{\xi,\alpha}(t_k)}\left( \frac{\|x_k-x_{k-1}\|}{t_{k}-t_{k-1}}\right)^2,
$$
which, using \eqref{ns.KT} and the first inequality in \eqref{eq:itf} yields the desired  inequality.

To end the proof,  combine \eqref{eq:002} with the last inequality in Corollary~\ref{cr:kanttk}.
\end{proof}

Therefore, it follows from Corollaries \ref{cr:kanttk} and \ref{auxcor} that all statements in Theorem \ref{th:knt} are valid.

\section{Special cases} \label{sec:ec}
In this section, we present special cases for Theorem~\ref{th:knt}.
They include the case  where $x_0$ is a regular point of the inclusion  \eqref{eq:cc},  and the case where $x_0$  satisfies the Robinson condition. Moreover, we  present the result of convergence under the Lipschitz and Smale conditions.

\subsection{Convergence result for regular  starting point}
In this section we present a correspondent theorem to Theorem~\ref{th:knt}, namely,   we assume that  $x_0$ is a regular point of the inclusion \eqref{eq:cc}, see \cite{burke} and  references therein.
We also present results of convergence under the Lipschitz and Smale condition. We start by defining regularity.
\begin{definition}\label{regular2}
Let $F:\mathbb{R}^n \to \mathbb{R}^{m}$ be a continuously differentiable
function and  let $h:\mathbb{R}^{m}\to \mathbb{R}$ be a real-valued convex function with minimizer set $C$ nonempty.
A point  $x_0 \in \mathbb{R}^n$  is a regular point of the inclusion $F(x)\in C$ if
$$
Ker(F'(x_0)^T)\cap(C-F(x_0))^{o}=\{0\},
$$
\end{definition}
As we know (see \cite{chon10})  the  definition of a quasi-regular point extends the definition of a regular point.
The following proposition  relates these two concepts, where the existence
of constants r and $\beta$ is due to Burke and Ferris in \cite{burke}, and the second assertion then
follows from  Remark~\ref{rem:1}.
\begin{proposition}\label{regular}
Let $x_0 \in \mathbb{R}^n$ be a regular point of the inclusion $F(x)\in C$. Then there exist
constants $r>0$ and  $\beta>0$ such that
$$
D_C(x)\neq \emptyset \quad e \quad d(0,D_C(x))\leq\beta d(F(x),C),  \qquad \forall x \in B(x_0,r).
$$
 Consequently,
$x_0$ is a quasi-regular point with the quasi-regular radius $r_{x_0}\geq r$ and the quasi-regular bound function
$\beta_{x_0} (\cdot)\leq \beta$ on $[0,r)$, as defined in \eqref{eq:qrr} and \eqref{eq:qrf}, respectively.
\end{proposition}
From now on, for each   regular point $x_0 \in \mathbb{R}^n$  of the inclusion  $F(x)\in C$ we will denote  by $r>0$ and $\beta>0$, the associated  constants given by the last proposition.
\begin{theorem}\label{th:kntrp}
Let $F:\mathbb{R}^n \to \mathbb{R}^{m}$ be a continuously differentiable
function. Assume that $R>0$, $x_0 \in \mathbb{R}^n$ and $f:[0,\; R)\to \mathbb{R}$ is a
 majorant function for $F$ on  $B(x_0,R).$ Take the constants $\alpha>0$ and
 $\xi>0$ and consider the
auxiliary function $f_{ \xi, \alpha}:[0,R)\to \mathbb{R}$,
$$
  f_{ \xi, \alpha}(t)=\xi+(\alpha-1)t+\alpha f(t).
$$
 If,  $f_{ \xi, \alpha}$ satisfies {\bf h3}, i.e., $t_*$ is the smallest zero of $f_{ \xi, \alpha},$  then  the sequence
  generated by Newton's Method  for solving $ f_{ \xi, \alpha}(t)=0,$  with starting
  point $t_{0}=0$,
$$
    t_{k+1} ={t_{k}}-f_{ \xi, \alpha}'(t_{k}) ^{-1}f_{ \xi, \alpha}(t_{k}),\quad k=0,1,\ldots\,,
$$
  is well defined, $\{t_{k}\}$ is strictly increasing, is contained in
  $[0,t_*)$, and converges $Q$-linearly to $t_*$.
Let  \(\eta \in [1,\infty)\),  \( \Delta \in (0,\infty]\) and $h:\mathbb{R}^{m}\to \mathbb{R}$  be a real-valued convex function
with minimizer set $C$ nonempty. Suppose that $x_0 \in \mathbb{R}^n$ is a
regular point of the inclusion $F(x)\in C$ with  associated constants $r>0$ and $\beta>0$.  If    \(d(F(x_0),C) >0\),   $t_* \leq r$,
$$
  {\Delta}\geq \xi \geq \eta \beta d(F(x_0),C), \qquad \alpha\geq {\eta\beta}/({\eta\beta[f'({\xi})+1]+1}),
$$
 then the sequence generated by {\bf Algorithm \ref{algor1}}, denoted by $\{x_k\},$ is contained in $B(x_0,
  t_*)$,
$$
F(x_k)+F'(x_k)(x_{k+1}-x_{k}) \in C,  \quad k=0,1, \ldots\, ,
$$
satisfies the inequalities
$$
  \|x_{k+1}-x_{k}\|   \leq  t_{k+1}-t_{k} , \qquad \|x_{k+1}-x_{k}\|\leq   \frac{t_{k+1}-t_{k}}{(t_{k}-t_{k-1})^2} \|x_k-x_{k-1}\|^2,
$$
for \( k=0, 1, \ldots\, , \) and \( k=1,2, \ldots\, \), respectively,  converging  to a point $x_*\in B[x_0, t_*]$ such that  $F(x_*)\in C,$
$$
 \|x_*-x_{k}\|   \leq  t_*-t_{k}, \qquad k=0,1, \ldots\,
$$
and  the convergence is  $R$-linear. If, additionally,   $f_{\xi, \alpha}$ satisfies  {\bf h4}    then the sequences $\{t_k\}$ and  $\{x_k\}$ converge $Q$-quadratically  and $R$-quadratically
 to $t_*$ and $x_*$, respectively.
\end{theorem}
\begin{proof}
Since  $x_0$ is a regular point for the inclusion, we have from Proposition~\ref{regular}  that   $x_0$ is a quasi-regular point for the inclusion
 $F(x)\in C$ with the quasi-regular radius   $r_{x_0}\geq r$. So, taking into account the assumption $t_*\leq r$ we obtain
$$t_*<r_{x_0}.$$
Moreover,  Proposition~\ref{regular} also implies that the  quasi-regular bound function
\begin{equation}\label{212121}
\beta_{x_0}(t)\leq \beta, \qquad \forall~t~\in [0,r).
\end{equation}
Since $ \Delta \geq \xi\geq \eta\beta d(F(x_0),C)$ and   the last inequality implies that   $\beta_{x_0}(0)\leq \beta$,  we have
$$
  \Delta \geq \xi\geq \eta\beta_{x_0}(0)d(F(x_0),C).
 $$
Now, combining  the assumptions  $0<\xi$ and $t_* \leq r$ with the  first statement in Proposition~\ref{nnn}  we conclude that $0<\xi<t_* \leq r$.
 So,  using \eqref{212121}, $f'(0)=-1$,  $f'$ as strictly increasing and $\eta \geq1$; after simple algebraic manipulation we obtain
$$
\frac{\eta\beta}{\eta\beta[f'({\xi})+1]+1}\geq \frac{\eta\beta_{x_0}(t)}{\eta\beta_{x_0}(t)[f'(t)+1]+1}, \qquad \forall~t~\in~[\xi,t_*).
$$
Hence, the assumption $ \alpha \geq {\eta\beta}/({\eta\beta[f'({\xi})+1]+1})$ and the last inequality imply that
$$
 \alpha \geq  \sup\left\{\frac{\eta\beta_{x_0}(t)}{\eta\beta_{x_0}(t)[f'(t)+1]+1}:{\xi}\leq t < t_*\right\}.
$$
Therefore,   $F$ and $x_0$ satisfy all assumptions in Theorem~\ref{th:knt}   and consequently the statements of the theorem are satisfied.
\end{proof}
Under  the Lipschitz condition, Theorem \ref{th:kntrp} becomes:
\begin{theorem}\label{th:kntrplip}
Let $F:\mathbb{R}^n \to \mathbb{R}^{m}$ be a continuously differentiable
function. Assume that $x_0 \in \mathbb{R}^n$, $R>0$ and $K>0,$ such that
 $$
   \|F'(y)-F'(x)\| \leq
   K\|y-x\|,\qquad x,y \; \in B(x_0,R).
$$
Take the constants $\alpha>0$ and
 $\xi>0$ and consider the
auxiliary function $f_{ \xi, \alpha}:[0,R)\to \mathbb{R}$,
$$
  f_{ \xi, \alpha}(t)=\xi-t+ (\alpha Kt^2)/2.
$$
 If $ 2\alpha  K \xi \leq1$, then $t_*=(1-\sqrt{1-2\alpha K \xi}\, )/({\alpha K})$ is the smallest zero of $f_{ \xi, \alpha},$   the sequence
  generated by Newton's Method  for solving $ f_{ \xi, \alpha}(t)=0,$  with starting
  point $t_{0}=0$,
$$
    t_{k+1} ={t_{k}}-f_{ \xi, \alpha}'(t_{k}) ^{-1}f_{ \xi, \alpha}(t_{k}),\quad k=0,1,\ldots\,,
$$
 is well defined, $\{t_{k}\}$ is strictly increasing, is contained in
  $[0,t_*)$, and converges $Q$-linearly to $t_*$.
Let  \(\eta \in [1,\infty)\),  \( \Delta \in (0,\infty]\) and $h:\mathbb{R}^{m}\to \mathbb{R}$  be a real-valued convex function
with minimizer set $C$ nonempty. Suppose that $x_0 \in \mathbb{R}^n$ is a
regular point of the inclusion $F(x)\in C$ with  associated constants $r>0$ and $\beta>0$.  If    \(d(F(x_0),C) >0\),   $t_* \leq r$,
$$
  {\Delta}\geq \xi \geq \eta \beta d(F(x_0),C), \qquad \alpha \geq {\eta\beta}/({K\eta\beta\xi+1}),
$$
 then the sequence generated by {\bf Algorithm \ref{algor1}}, denoted by $\{x_k\},$ is contained in $B(x_0,
  t_*)$,
$$
F(x_k)+F'(x_k)(x_{k+1}-x_{k}) \in C,  \quad k=0,1, \ldots\, ,
$$
satisfies the inequalities
$$
  \|x_{k+1}-x_{k}\|   \leq  t_{k+1}-t_{k} , \qquad \|x_{k+1}-x_{k}\|\leq   \frac{t_{k+1}-t_{k}}{(t_{k}-t_{k-1})^2} \|x_k-x_{k-1}\|^2,
$$
for \( k=0, 1, \ldots\, , \) and \( k=1,2, \ldots\, \), respectively,  converging  to a point $x_*\in B[x_0, t_*]$ such that  $F(x_*)\in C,$
$$
 \|x_*-x_{k}\|   \leq  t_*-t_{k}, \qquad k=0,1, \ldots\,
$$
and  the convergence is  $R$-linear. If, additionally,   $ 2\alpha  K \xi <1$   then the sequences $\{t_k\}$ and  $\{x_k\}$ converge $Q$-quadratically  and $R$-quadratically
 to $t_*$ and $x_*$, respectively.
\end{theorem}
\begin{proof}
It is promptly proved that   $f:[0, R)\to \mathbb{R}$ defined by $ f(t)=Kt^{2}/2-t$ is a majorant function for the function F on $B(x_0, R)$.  Hence,
$$
  f_{ \xi, \alpha}(t)=\xi-t+ (\alpha Kt^2)/2=\xi+(\alpha-1)t+\alpha f(t),
$$
 and, since  $ 2\alpha  K \xi \leq1$,   we conclude that $ f_{ \xi, \alpha}$ satisfies {\bf h3}  and   $t_*=(1-\sqrt{1-2\alpha K \xi}\, )/({\alpha K})$ is its smallest root.
In this case,  the constant $\alpha$ satisfies
$$
\alpha\geq \frac{\eta\beta}{1+K\eta\beta\xi}= \frac{\eta\beta}{\eta\beta[f'({\xi})+1]+1} .
$$
Therefore, taking $\alpha$, $ f_{\xi,\alpha}$ and $t_*$  as defined above, all  the statements of the  first part of the theorem follow from Theorem~\ref{th:kntrp}. For proving the second part,  it is sufficient to note that the assumption $ 2\alpha  K \xi <1$ implies that $ f_{ \xi, \alpha}$ satisfies {\bf h4}.
\end{proof}
Under the Smale condition, see \cite{S86}, Theorem \ref{th:kntrp} becomes:

\begin{theorem}\label{th:kntrpsma}
Let $F:\mathbb{R}^n \to \mathbb{R}^{m}$ be an analytic
function. Assume that $x_0 \in \mathbb{R}^n$  and
\begin{equation} \label{eq:SmaleCond}
  \gamma := \sup _{ n > 1 }\left\| \frac
{F^{(n)}(x_0)}{n !}\right\|^{1/(n-1)}<+\infty.
\end{equation}
Take the constants $\alpha>0$ and
 $\xi>0$ and consider the auxiliary function ${f_{\xi,\alpha}}:[0,1/\gamma)\to \mathbb{R}$,
$$
 {f_{\xi,\alpha}}(t)=\frac{\alpha\gamma}{1-\gamma t}t^2-t+\xi.
$$
If  $  \xi\gamma\leq 1+2\alpha-2\sqrt{\alpha(1+\alpha)}$ then
$$
t_*=\frac{1+\gamma\xi-\sqrt{(1+\gamma\xi)^2-4(1+\alpha)\gamma\xi}}{2(1+\alpha)\gamma},
$$
 is the smallest zero of $f_{ \xi, \alpha},$   the sequence
  generated by Newton's Method  for solving $ f_{ \xi, \alpha}(t)=0,$  with starting
  point $t_{0}=0$,
$$
    t_{k+1} ={t_{k}}-f_{ \xi, \alpha}'(t_{k}) ^{-1}f_{ \xi, \alpha}(t_{k}),\quad k=0,1,\ldots\,,
$$
  is well defined, $\{t_{k}\}$ is strictly increasing, is contained in
  $[0,t_*)$, and converges $Q$-linearly to $t_*$.
Let  \(\eta \in [1,\infty)\),  \( \Delta \in (0,\infty]\) and $h:\mathbb{R}^{m}\to \mathbb{R}$ be  a real-valued convex function
with minimizer set $C$ nonempty. Suppose that $x_0 \in \mathbb{R}^n$ is a
regular point of the inclusion $F(x)\in C$ with  associated constants $r>0$ and $\beta>0$.  If    \(d(F(x_0),C) >0\),   $t_* \leq r$,
$$
  {\Delta}\geq \xi \geq \eta \beta d(F(x_0),C), \qquad \alpha \geq \frac{\eta\beta(1-\gamma\xi)^2}{\eta\beta+(1-\eta\beta)(1-\gamma\xi)^2},
$$
 then the sequence generated by {\bf Algorithm \ref{algor1}}, denoted by $\{x_k\},$ is contained in $B(x_0,
  t_*)$,
$$
F(x_k)+F'(x_k)(x_{k+1}-x_{k}) \in C,  \quad k=0,1, \ldots\, ,
$$
satisfies the inequalities
$$
  \|x_{k+1}-x_{k}\|   \leq  t_{k+1}-t_{k} , \qquad \|x_{k+1}-x_{k}\|\leq   \frac{t_{k+1}-t_{k}}{(t_{k}-t_{k-1})^2} \|x_k-x_{k-1}\|^2,
$$
for \( k=0, 1, \ldots\, , \) and \( k=1,2, \ldots\, \), respectively,  converging  to a point $x_*\in B[x_0, t_*]$ such that  $F(x_*)\in C,$
$$
 \|x_*-x_{k}\|   \leq  t_*-t_{k}, \qquad k=0,1, \ldots\,
$$
and  the convergence is  $R$-linear. If, additionally,   $  \xi\gamma< 1+2\alpha-2\sqrt{\alpha(1+\alpha)}$   then the sequences $\{t_k\}$ and  $\{x_k\}$ converge $Q$-quadratically  and $R$-quadratically
 to $t_*$ and $x_*$, respectively.
\end{theorem}
We need the following results to prove the above theorem.
\begin{lemma} \label{lemma:qc1}
Let $F:\mathbb{R}^n \to \mathbb{R}^{m}$ be an analytic function.
 Suppose that
$x_0\in \mathbb{R}^n $  and  $\gamma$ is defined in
\eqref{eq:SmaleCond}. Then, for all $x\in B(x_{0}, 1/\gamma)$
it holds that
$$
\|F''(x)\| \leqslant  (2\gamma)/( 1- \gamma \|x-x_0\|)^3.
$$
\end{lemma}
\begin{proof}
The proof follows the same pattern as Lemma 21 from \cite{Max2}.
\end{proof}
\begin{lemma} \label{lc}
Let $F:\mathbb{R}^n \to \mathbb{R}^{m}$ be twice continuously differentiable.  If there exists  a \mbox{$f:[0,R)\to \mathbb {R}$} twice continuously differentiable and satisfying
$$
\|F''(x)\|\leqslant f''(\|x-x_0\|),
$$
for all $x\in \mathbb{R}^n$ such that  $\|x-x_0\|<R$, then $F$ and $f$ satisfy \eqref{KH}.
\end{lemma}
\begin{proof}
The proof follows the same pattern as Lemma 22 from \cite{Max2}.
\end{proof}

\noindent
{\bf Proof of Theorem \ref{th:kntrpsma}.}
Consider the real function   $f:[0,1/\gamma) \to \mathbb{R}$ defined by
$$
f(t)=\frac{t}{1-\gamma t}-2t.
$$
It is straightforward to show that $f$ is  analytic and that
$$
f(0)=0, \quad f'(t)=1/(1-\gamma t)^2-2, \quad f'(0)=-1, \quad
f''(t)=(2\gamma)/(1-\gamma t)^3, \quad f^{n}(0)=n!\,\gamma^{n-1},
$$
for $n\geq 2$. It follows from the last equalities that f satisfies {\bf h1} and
{\bf h2} in Definition \ref{majo}. Now, as $f''(t)=(2\gamma)/(1-\gamma t)^3,$ combining the
Lemmas \ref{lemma:qc1} and   \ref{lc}, we have   $F$ and $f$ satisfy  \eqref{KH} with
 $R=1/\gamma.$ Therefore, f is a majorant function for $F$ on $B(x_0,1/\gamma)$.
 Hence,
 $$
 f_{\xi,\alpha}(t)=\frac{\alpha\gamma}{1-\gamma t}t^2-t+\xi=\xi+(\alpha-1)t+\alpha f(t),
 $$
and, since   $  \xi\gamma\leq 1+2\alpha-2\sqrt{\alpha(1+\alpha)}$, we conclude that $ f_{ \xi, \alpha}$ satisfies {\bf h3}  and  $$t_*=\frac{1+\gamma\xi-\sqrt{(1+\gamma\xi)^2-4(1+\alpha)\gamma\xi}}{2(1+\alpha)\gamma}$$  is its smallest root.
In this case,  the constant $\alpha$ satisfies
$$
\alpha\geq \frac{\eta\beta(1-\gamma\xi)^2}{\eta\beta+(1-\eta\beta)(1-\gamma\xi)^2}= \frac{\eta\beta}{\eta\beta[f'({\xi})+1]+1} .
$$
Therefore, taking $\alpha$, $ f_{\xi,\alpha}$ and $t_*$  as defined above, all  the statements of the  first part of the theorem follow from Theorem~\ref{th:kntrp}.
For proving the second part,  it is sufficient to note that the assumption  $  \xi\gamma < 1+2\alpha-2\sqrt{\alpha(1+\alpha)}$ implies that $ f_{ \xi, \alpha}$
satisfies {\bf h4}.\qed

\subsection{Convergence result under the Robinson condition }

In this section we present a correspondent theorem to Theorem~\ref{th:knt}, namely,   we assume that  $x_0$ satisfies the Robinson condition, see  \cite{chon10} and \cite{robinson1}.   Under the Robinson condition,  we also present results of convergence for the Lipschitz and Smale conditions. We start by defining  the Robinson condition.

Let $C\subset \mathbb{R}^m$ be a nonempty closed convex cone, $F:\mathbb{R}^n \to \mathbb{R}^{m}$ be a continuously differentiable
function and $x\in \mathbb{R}^x$. Define the multifunction  $T_{x}: \mathbb{R}^n \to  P( \mathbb{R}^m)$ as
\begin{equation}\label{ro2}
T_{x}d=F'(x)d-C.
\end{equation}
The multifunction $T_{x}$ is a convex process from $\mathbb{R}^n$ to  $\mathbb{R}^m$. Convex process has been extensively studied in \cite{Rocka2,Rocka}.  As usual, the domain, norm and inverse of  $T_{x}$ are defined, respectively, by
$$
 {\cal D}(T_{x}):=\{d\in \mathbb{R}^n:T_{x}d \neq \emptyset \}, \qquad \|T_{x}\|:=sup \; \{ \|T_{x} d\|: x\in {\cal D}(T_{x}), \; \|d\| \leq 1 \},
$$
$$
T_{x}^{-1}y:=\{d\in \mathbb{R}^n:F'(x)d\in y+C\}, \qquad    \; y \in \mathbb{R}^m.
$$
where $ \|T_{x} d\|:=\inf \{\|v\|~: ~v\in T_{x } d  \}$.

The point $x_0\in \mathbb{R}^n$  satisfies the   {\it Robinson  condition}  if  the multifunction  $T_{x_0}$  carries $\mathbb{R}^n$  onto $\mathbb{R}^m$, that is,
 \begin{equation} \label{robinson}
\forall ~ y\in \mathbb{R}^m  \quad  \exists ~ d\in \mathbb{R}^n, ~ \exists ~ c\in C;  \quad y=F'(x_0)d-c.
\end{equation}
\begin{theorem}\label{th:kntcr}
Let $F:\mathbb{R}^n \to \mathbb{R}^{m}$ be a continuously differentiable
function. Assume that $R>0$, $x_0 \in \mathbb{R}^n$ and $f:[0,\; R)\to \mathbb{R}$ is a
 majorant function for $F$ on  $B(x_0,R).$ Take the constants $\alpha>0$ and
 $\xi>0$ and consider the
auxiliary function $f_{ \xi, \alpha}:[0,R)\to \mathbb{R}$,
$$
  f_{ \xi, \alpha}(t)=\xi+(\alpha-1)t+\alpha f(t).
$$
 If,  $f_{ \xi, \alpha}$ satisfies {\bf h3}, i.e., $t_*$ is the smallest zero of $f_{ \xi, \alpha},$  then   the sequence
  generated by Newton's Method  for solving $ f_{ \xi, \alpha}(t)=0,$  with starting
  point $t_{0}=0$,
$$
    t_{k+1} ={t_{k}}-f_{ \xi, \alpha}'(t_{k}) ^{-1}f_{ \xi, \alpha}(t_{k}),\quad k=0,1,\ldots\,,
$$
 is well defined, $\{t_{k}\}$ is strictly increasing, is contained in
  $[0,t_*)$, and converges $Q$-linearly to $t_*$.
Let  \(\eta \in [1,\infty)\),  \( \Delta \in (0,\infty]\) and $h:\mathbb{R}^{m}\to \mathbb{R}$  be a real-valued convex function
with minimizer set $C$ nonempty. Suppose that $C$ is a cone and $x_0 \in \mathbb{R}^n$ satisfies the Robinson condition. Let   $\beta_0=\|T_{x_0}^{-1}\|$. If    \(d(F(x_0),C) >0\), \( t_*\leq r_{\beta_0}:=\{t\in [0,R)~:~ \beta_0-1 +\beta_0f'(t)<0\}\),
 $$
 \Delta \geq \xi\geq \eta\beta_0 d(F(x_0),C), \qquad  \alpha \geq \frac{\eta\beta_0}{1+(\eta-1)\beta_0[f'(\xi)+1]},
 $$
 then the sequence generated by {\bf Algorithm \ref{algor1}}, denoted by $\{x_k\},$ is contained in $B(x_0,
  t_*)$,
$$
F(x_k)+F'(x_k)(x_{k+1}-x_{k}) \in C,  \quad k=0,1, \ldots\, ,
$$
satisfies the inequalities
$$
  \|x_{k+1}-x_{k}\|   \leq  t_{k+1}-t_{k}, \qquad \|x_{k+1}-x_{k}\|\leq   \frac{t_{k+1}-t_{k}}{(t_{k}-t_{k-1})^2} \|x_k-x_{k-1}\|^2,
$$
for \( k=0, 1, \ldots\, , \) and \( k=1,2, \ldots\, \), respectively,  converging  to a point $x_*\in B[x_0, t_*]$ such that  $F(x_*)\in C,$
$$
 \|x_*-x_{k}\|   \leq  t_*-t_{k}, \qquad k=0,1, \ldots\,
$$
and  the convergence is  $R$-linear. If, additionally,   $f_{\xi, \alpha}$ satisfies  {\bf h4}    then the sequences $\{t_k\}$ and  $\{x_k\}$ converge $Q$-quadratically  and $R$-quadratically
 to $t_*$ and $x_*$, respectively.
\end{theorem}
We need the following two results to prove the above theorem.

\begin{lemma}\label{robinson210}
Let $F:\mathbb{R}^n \to \mathbb{R}^{m}$ be a continuously differentiable
function and  $C$ a nonempty closed convex cone. Suppose that  $x_0 \in \mathbb{R}^n$ satisfies the Robinson condition.  Then
 $$
 \|T_{x_0}^{-1}\|<+\infty.
 $$
Moreover, if S is a linear transformation from $\mathbb{R}^n$ to $ \mathbb{R}^m$ such that $\|T_{x_0}^{-1}\|\|S\|< 1$, then the convex process $\bar{T}$, defined by $\bar{T}:=T_{x_0}+S$,  carries $\mathbb{R}^n$ onto $ \mathbb{R}^m$, $\|\bar{T}^{-1}\|<+\infty$ and
$$\|\bar{T}^{-1}\|\leq\frac{\|T_{x_0}^{-1}\|}{1-\|T_{x_0}^{-1}\|\|S\|}.$$
\end{lemma}
\begin{proof}
See Theorem~1 on p.342 of \cite{robinson1}.
\end{proof}
\begin{lemma}\label{robinson2}
Let $F:\mathbb{R}^n \to \mathbb{R}^{m}$ be a continuously differentiable
function and let  $h:\mathbb{R}^{m}\to \mathbb{R}$  be a real-valued convex function with minimizer set $C$ nonempty. Suppose that  $x_0 \in \mathbb{R}^n$ satisfies the Robinson condition.  Then  $x_0$ is a regular point of the inclusion  $F(x)\in C$, and in particular, $x_0$ is a quasi-regular point of the inclusion  $F(x)\in C$. Moreover,  assume  $C$ is a cone,  $R>0 $ and  $f:[0,\; R)\to \mathbb{R}$  is a  majorant  function for  $F$ on  $B(x_0,R)$. Let  $\xi>0$,  $\beta_0=\|T_{x_0}^{-1}\|$, the auxiliary function $f_{ \xi, \beta_0}:[0,R)\to \mathbb{R}$,
$$
  f_{ \xi, \beta_0}(t):=\xi+(\beta_0-1)t+\beta_0 f(t),
$$
and  $r_{\beta_0}:=\sup\{t\in [0,R): f_{\xi,\beta_0}'(t)<0\}$. If  $r_{x_0}$ is the quasi-regular radius and $\beta_{x_0}(\cdot)$ is the quasi-regular bound  function  for  the quasi-regular point  $x_0$,  then
$$
r_{x_0}\geq r_{\beta_0}, \qquad \beta_{x_0}(t)\leq \frac{\beta_0}{1-\beta_0[f'(t)+1]}, \qquad \forall \;  t \in  [0, ~r_{\beta_0}).
$$
\end{lemma}
\begin{proof}
Take $y\in Ker(F'(x_0)^T)\cap(C-F(x_0))^{o}.$ Hence,
$$
0=\langle F'(x_0)^Ty,d\rangle =\langle y,F'(x_0)d\rangle, \quad  \forall~ d\in \mathbb{R}^n, \qquad  \qquad \langle y,c-F(x_0)\rangle\leq0, \quad  \forall~ c \in C.
$$
Since $x_0$ satisfies the Robinson condition,
$d\in \mathbb{R}^n$ and $c\in C$ exist, such that $-y-F(x_0)=F'(x_0)d-c$, which  combining with the above   inequalities gives
$$
\langle y,y \rangle =\langle y,c-F(x_0)-F'(x_0)d\rangle =\langle d,c-F(x_0)\rangle \leq 0.
$$
So $y=0$, and we obtain from Definition \ref{regular2} that $x_0$ is a regular point of the inclusion $F(x) \in C.$

To establish the second part, first take $x\in \mathbb{R}^n$ such that $\|x-x_0\|\leq r_{\beta_0}$. Using $f$ as a  majorant  function of  $F$ on  $B(x_0,R),$ as well as the definitions of  $\beta_0$,  $f_{ \xi, \beta_0}$ and $r_{\beta_0}$, we obtain
\begin{equation}\label{10101}
\|T_{x_0}^{-1} \| \|F'(x)-F'(x_0)\| \leq \beta_0[f'(\|x-x_0\|)-f'(0)]= f'_{ \xi, \beta_0}(\|x-x_0\|)+1<1.
\end{equation}
Using that  $x_0$  satisfies the Robinson condition and the last inequality,  it follows from Lemma~\ref{robinson210}  that the convex process
$$
T_{x}d=F'(x)d-C=T_{x_0}d+[F'(x)-F'(x_0)]d, \qquad \forall ~d\in \mathbb{R}^n,
$$
carries $\mathbb{R}^{n}$ onto $\mathbb{R}^m$ and
\begin{equation}\label{2121}
\|{T_x}^{-1}\|\leq\frac{\|T_{x_0}^{-1}\|}{1-\|T_{x_0}^{-1}\|\|F'(x)-F'(x_0)\|}\leq\frac{\beta_0}{1-\beta_0[f'(\|x-x_0\|)-f'(0)]},
\end{equation}
where the last inequality follows the definition of $\beta_0$ and \eqref{10101}. Moreover,  as $T_x$ carries $\mathbb{R}^{n}$ onto $\mathbb{R}^m$, we also have
\begin{equation}\label{ma11}
D_C(x)=\{d\in \mathbb{R}^n:F(x)+F'(x)d \in C\} \neq \emptyset, \qquad \forall \; x\in B(x_0,  r_{\beta_0}).
\end{equation}
 Now,  let $d\in T_{x}^{-1}(c-F(x)).$ Using the definition of $T_{x}^{-1}$ it follows that
$$
F'(x)d\in c-F(x)+C=C-F(x),
$$
hence we conclude that   $F(x)+F'(x)d\in C$,  which combining with the definition of $D_C(x)$ yields
$$
T_{x}^{-1}(c-F(x))\subset D_C(x).
$$
Therefore,
$$
d(0,D_C(x))\leq \|T_{x}^{-1}(c-F(x))\|\leq \|T_{x}^{-1}\|\|c-F(x)\|, \qquad \forall c\in C.
$$
The last inequality together with \eqref{2121} imply
$$
d(0,D_C(x))\leq \|T_{x}^{-1}\|d(F(x),C)\leq\frac{\beta_0}{1-\beta_0[f'(\|x-x_0\|)-f'(0)]}d(F(x),C),
$$
which combined with \eqref{ma11}, as well as definitions of $r_{x_0}$ and $\beta_{x_0} (\cdot)$ in  \eqref{eq:qrr} and \eqref{eq:qrf}, respectively,
yields the desired inequalities.
\end{proof}

\noindent
{\bf [Proof of Theorem \ref{th:kntcr}]}  Since   $x_0 \in \mathbb{R}^n$ satisfies the Robinson condition, we have from Lemma~\ref{robinson2} that
$x_0$ is a quasi-regular point of the inclusion $F(x)\in C$  with the quasi-regular radius  $r_{x_0} \geq r_{\beta_0} $. So, taking into account the assumption $ t_*\leq r_{\beta_0}$ we obtain
$$
 t_*<r_{x_0}.
 $$
 Moreover,  Lemma~\ref{robinson2}  also implies that the  quasi-regular bound function $\beta_{x_0}(\cdot)$ satisfies
\begin{equation} \label{eq:ierc}
 \beta_{x_0}(t)\leq \frac{\beta_0}{1-\beta_0[f'(t)+1]}, \qquad \forall \;  t \in  [0,~r_{\beta_0}).
 \end{equation}
Since $ \Delta \geq \xi\geq \eta\beta_0 d(F(x_0),C)$ and   the last inequality implies that   $\beta_{x_0}(0)\leq \beta_0$,  we have
$$
  \Delta \geq \xi\geq \eta\beta_{x_0}(0)d(F(x_0),C).
 $$
Now, combining  the assumptions  $0<\xi$ and $t_* \leq r_{\beta_0}$ with the  first statement in Proposition~\ref{nnn}  we conclude that  $0<\xi < t_* \leq  r_{\beta_0}$.  So,  using \eqref{eq:ierc}, $f'(0)=-1$,  $f'$ as strictly increasing and $\eta \geq1$; after simple algebraic manipulation we obtain
$$
\eta[f'(t)+1]+\frac{1}{\beta_{x_0}(t)} \geq \frac{1}{\beta_0}+(\eta-1)[f'(t)+1]\geq \frac{1}{\beta_0}+(\eta-1)[f'(\xi)+1],  \qquad \forall~ t \in  [\xi,~ t_*),
$$
or equivalently,
\begin{equation} \label{eq:ax1}
\frac{\eta\beta_0}{1+(\eta-1)\beta_0[f'(\xi)+1]} \geq \frac{ \eta\beta_{x_0}(t)} { \eta\beta_{x_0}(t)[f'(t)+1]+1},  \qquad \forall~ t \in  [\xi,~ t_*).
\end{equation}
Hence, the assumption $ \alpha \geq \eta\beta_0/[1+(\eta-1)\beta_0(f'(\xi)+1)]$ and the last inequality imply that
$$
 \alpha \geq  \sup\left\{\frac{\eta\beta_{x_0}(t)}{\eta\beta_{x_0}(t)[f'(t)+1]+1}:{\xi}\leq t < t_*\right\}.
$$
Therefore,   $F$ and $x_0$ satisfy all assumptions in Theorem~\ref{th:knt}   and so statements of the theorem follows.
\qed
\begin{remark}
Let $F:\mathbb{R}^n \to \mathbb{R}^{m}$ be a continuously differentiable
function. Assume that $x_0 \in \mathbb{R}^n$, $R>0$ and  $K>0$ exists, such that
 $$
   \|F'(y)-F'(x)\| \leq
   K\|y-x\|,\qquad x,y \; \in B(x_0,R).
$$
Note that $f:[0, R)\to \mathbb{R}$ defined by $ f(t)=Kt^{2}/2-t$ is a majorant function for the function F on $B(x_0, R)$.  In this case, it  is easy to see that    $\bf{h3}$, $\bf{h4}$ and $t_*$  in Theorem~\ref{th:kntcr} become
$$
 2\alpha  K \xi \leq1, \qquad 2\alpha  K \xi <1, \qquad t_*=(1-\sqrt{1-2\alpha K \xi}\, )/({\alpha K}),
 $$
and $\alpha$ satisfies
$$
  \qquad\alpha\geq \frac{\eta\beta_0}{1+(\eta-1) K\beta_0 \xi}.
$$
In particular, if $C=\{0\}$ and $n=m$, the Robinson condition is equivalent to the condition that $F'(x_0)^{-1}$ is non-singular.
Hence, for $\eta=1$ we obtain the semi-local convergence for the Newton method under the Lipschitz condition,  see \cite{dennis}.
\end{remark}

\begin{remark}
Let $F:\mathbb{R}^n \to \mathbb{R}^{m}$ be an analytic
function. Assume that $x_0 \in \mathbb{R}^n$  and
\begin{equation} \label{eq:SmaleCond}
  \gamma := \sup _{ n > 1 }\left\| \frac
{F^{(n)}(x_0)}{n !}\right\|^{1/(n-1)}<+\infty.
\end{equation}
Note that the real function   $f:[0,1/\gamma) \to \mathbb{R}$ defined by
$
f(t)=t/(1-\gamma t)-2t
$
is a majorant function for the function F on $B(x_0, 1/\gamma)$.   In this case, it  is easy see that    $\bf{h3}$, $\bf{h4}$ and $t_*$  in Theorem~\ref{th:kntcr} become
$$
 \xi\gamma\leq 1+2\alpha-2\sqrt{\alpha(1+\alpha)}, \qquad \xi\gamma < 1+2\alpha-2\sqrt{\alpha(1+\alpha)},
$$
$$t_*=\frac{1+\gamma\xi-\sqrt{(1+\gamma\xi)^2-4(1+\alpha)\gamma\xi}}{2(1+\alpha)\gamma},$$
and $\alpha$ satisfies
$$\alpha\geq\frac{\eta\beta_0(1-\gamma\xi)^2}{(\eta-1)\beta_0+[1-\beta_0(\eta-1)](1-\gamma\xi)^2}.$$
In particular, if $C=\{0\}$ and $n=m$, the Robinson condition is equivalent to the condition that $F'(x_0)^{-1}$ is non-singular.
Hence, for $\eta=1$ we obtain the semi-local convergence for the Newton method under the Smale condition,  see \cite{S86}.

\end{remark}
\bibliographystyle{abbrv}

\begin{thebibliography}{10}


\bibitem{burke}
J.~Burke and M.~C. Ferris.
\newblock A {G}auss-{N}ewton method for convex composite optimization.
\newblock Technical report, 1993.

\bibitem{dennis}
J.~E. Dennis, Jr. and R.~B. Schnabel.
\newblock {\em Numerical methods for unconstrained optimization and nonlinear
  equations (Classics in Applied Mathematics, 16)}.
\newblock Soc for Industrial \& Applied Math, 1996.

\bibitem{Max2}
O.~P. Ferreira, M.~L.~N. Goncalves, and P.~R. Oliveira.
\newblock Local convergence analysis of the {G}auss-{N}ewton method under a
  majorant condition.
\newblock {\em J. Complexity}, 27(1):111 -- 125, 2011.

\bibitem{FS02}
O.~P. Ferreira and B.~F. Svaiter.
\newblock Kantorovich's theorem on {N}ewton's method in riemannian manifolds.
\newblock {\em J. Complexity}, 18(1):304 -- 329, 2002.

\bibitem{FS01}
O.~P. Ferreira and B.~F. Svaiter.
\newblock Kantorovich's majorants principle for {N}ewton's method.
\newblock {\em Comput. Optim. Appl.}, 42(2):213--229, 2009.

\bibitem{Lem1}
J.-B. Hiriart-Urruty and C.~Lemar{\'e}chal.
\newblock {\em Convex analysis and minimization algorithms. {I}}, volume 305 of
  {\em Grundlehren der Mathematischen Wissenschaften [Fundamental Principles of
  Mathematical Sciences]}.
\newblock Springer-Verlag, Berlin, 1993.
\newblock Fundamentals.

\bibitem{chon10}
C.~Li and K.~F. Ng.
\newblock Majorizing functions and convergence of the {G}auss-{N}ewton method for
  convex composite optimization.
\newblock {\em SIAM J. Optim}, pages 613--642, 2007.

\bibitem{chon2002}
C.~Li and X.~Wang.
\newblock On convergence of the {G}auss-{N}ewton method for convex composite
  optimization.
\newblock {\em Math. Program.}, 91:349--356, 2002.
\newblock 10.1007/s101070100249.

\bibitem{robinson1}
S.~M. Robinson.
\newblock Extension of {N}ewton's method to nonlinear functions with values in a
  cone.
\newblock {\em Numer. Math.}, 19:341--347, 1972.
\newblock 10.1007/BF01404880.

\bibitem{Rocka2}
R.~T. Rockafellar.
\newblock {\em Monotone processes of convex and concave type}.
\newblock Memoirs of the American Mathematical Society, No. 77. American
  Mathematical Society, Providence, R.I., 1967.

\bibitem{Rocka}
R.~T. Rockafellar.
\newblock {\em Convex analysis}.
\newblock Princeton Mathematical Series, No. 28. Princeton University Press,
  Princeton, N.J., 1970.

\bibitem{S86}
S.~Smale.
\newblock {N}ewton's method estimates from data at one point.
\newblock In {\em The merging of disciplines: new directions in pure, applied,
  and computational mathematics ({L}aramie, {W}yo., 1985)}, pages 185--196.
  Springer, New York, 1986.

\bibitem{XW10}
X.~Wang.
\newblock Convergence of {N}ewton's method and uniqueness of the solution of
  equations in {B}anach space.
\newblock {\em IMA J. Numer. Anal.}, 20(1):123--134, 2000.
\end{thebibliography}

\end{document}